\documentclass[11pt]{article}

\usepackage{amsmath, amssymb, amsthm, amsfonts, mathtools}
\usepackage{xcolor}
\usepackage{hyperref}
\usepackage{cite}

\definecolor{webgreen}{rgb}{0,.5,0}
\definecolor{webbrown}{rgb}{.6,0,0}

\newtheorem{theorem}{Theorem}
\newtheorem*{Notation}{Notation}
\newtheorem{corollary}[theorem]{Corollary}
\newtheorem{lemma}[theorem]{Lemma}
\newtheorem{proposition}[theorem]{Proposition}
\newtheorem{remark}[theorem]{Remark}

\DeclareMathOperator{\NR}{NR}
\newcommand{\Mod}[1]{\ (\mathrm{mod}\ #1)}

\title{\LARGE \bf A Unified Approach to Calculating Sylvester Sums}

\author{
\begin{tabular}{c}
Neha Gupta\thanks{\href{mailto:neha.gupta@snu.edu.in}{neha.gupta@snu.edu.in}} \\
Manoj Upreti\thanks{\href{mailto:mu506@snu.edu.in}{mu506@snu.edu.in}} \\
\small Department of Mathematics, Shiv Nadar Institution of Eminence \\
\small Greater Noida–201314, India
\end{tabular}
}

\date{}

\begin{document}

\maketitle

\vskip .2 in
\begin{abstract}
 In the context of the Frobenius coin problem, given two relatively prime positive integers $a$ and $b$, the set of nonrepresentable numbers consists of positive integers that cannot be expressed as nonnegative integer combination of $a$ and $b$. This work provides a formula for calculating the power sums of all nonrepresentable numbers, also known as the Sylvester sums. Although alternative formulas exist in the literature, our approach is based on an elementary observation. We consider the set of natural numbers from $1$ to $ab - 1$ and compute their total sum in two distinct ways, which leads naturally to the desired Sylvester sums. This method connects an analytic identity with a combinatorial viewpoint, giving a new way to understand these classical quantities.
 Furthermore, in this paper, we establish a criterion using the division algorithm to determine whether a given positive integer is nonrepresentable.
.

\end{abstract}

\section {Introduction}
Given relatively prime positive integers $a$ and $b$, a nonnegative integer $n$ is a representable number if it can be expressed as $ax+by$ for nonnegative integers $x$ and $y$. Conversely, if $n$ cannot be expressed in this form, it is called a nonrepresentable number. The set of nonrepresentable numbers, denoted by $\NR(a,b)$, is defined as:  
$$\NR(a,b)=\{n\in \mathbb{N}:n\neq ax+by; \text{ for any } x, y\in \mathbb{Z}_{\geq 0} \}.$$
Similarly, for $n$ variables $a_1,a_2,\dots,a_n$ with $\gcd(a_1,a_2,\dots,a_n)=1$, the set $\NR(a_1,a_2,\dots,a_n)$ is defined as
$$\NR(a_1,a_2,\dots,a_n)=\{n\in \mathbb{N}:n\neq a_1x_1+a_2x_2+\cdots+a_nx_n; \text{ for any } x_1,x_2,\dots,x_n\in \mathbb{Z}_{\geq 0} \}.$$

The well-known Frobenius coin problem involves finding the largest element in the set $\NR(a_1,a_2,\dots,a_n)$, known as the Frobenius number and denoted as $g(a_1,a_2,\dots,a_n)$. The cardinality of the set $\NR(a_1,a_2,\dots,a_n)$, is known as the Sylvester number and is denoted as $n(a_1,a_2,\dots,a_n)$.  In this paper, we restrict ourselves to the case $n=2$. The problem traces back to Sylvester, who is credited with the following elegant result (refer to \cite[Note~2, p.~17]{BR07})
$$g(a,b)=ab-a-b,$$
by establishing what is sometimes referred to as Sylvester's Lemma:
\begin{lemma}[Sylvester (1882)]
\label{SY1}
    For relatively prime positive integers $a$ and $b$, the equation $ax+by=n$ has a solution with $x,y\in \mathbb{Z}_{\geq 0}$, whenever $n\geq (a-1)(b-1)$. 
\end{lemma}

Sylvester also proved the following result:
$$n(a,b)=\frac{1}{2}(a-1)(b-1),$$
by establishing Theorem \ref{SY2} below\cite{JJS2}: 
\begin{theorem}[Sylvester (1882)]
\label{SY2}
      For relatively prime positive integers $a$ and $b$, the number of positive integers that cannot be expressed in the form $ax+by$ with $x,y\in \mathbb{Z}_{\geq 0}$ is equal to $\frac{1}{2}(a-1)(b-1).$ 
\end{theorem}

Computing the Frobenius number becomes considerably more challenging when extended beyond two variables and continues to be an active area of research. In 2017, Tripathi \cite{ATF} proposed a formula for the three-variable case; while insightful, its complexity limits its practical applicability in more general settings. However, explicit results are available under certain special conditions, for example, when numbers form an arithmetic sequence \cite{R56, Tri13} or a geometric sequence \cite{OP08, Tri08}. Additional contributions can be found in~\cite{KLP23, Tri16, Bat15}. 
For some recent developments, see~\cite{KoUG, KoC, KGG}, and for a comprehensive background and survey of the topic, refer to~\cite{RA07}.

 The study of power sums of elements of $\NR(a,b)$, often referred to as Sylvester sums, is another crucial aspect of this area of research \cite{BS,JR,MQ,AT,HJT}. These sums are denoted by $S_{m}(a,b)$, where $m$ is a nonnegative integer, and defined as:
$$S_m(a,b)=\sum_{n\in \NR(a,b)}n^m.$$
For $m=0$, it follows from Theorem \ref{SY2} that
\begin{equation}
\label{snm}
S_0(a,b)=\frac{1}{2}(a-1)(b-1).
\end{equation}

Brown and Shiue \cite{BS} calculated $S_1(a,b)$ by deriving a generating function of the characteristic function of the set $\NR(a,b)$. Later, Tripathi \cite{AT} provided a more direct approach to computing $S_1(a,b)$, first generating an explicit list of all nonrepresentable numbers using modular arithmetic and then summing them. He extended this method to compute $S_1(a_1,a_2,\dots,a_k)$. R{\o}dseth \cite{JR} obtained a closed form for $S_m(a,b)$ valid for every nonnegative integer $m$, using an exponential generating function.

Tuenter \cite{HJT} calculated $S_m(a,b)$ by deriving an identity that completely characterizes the set $\NR(a,b)$. In recent work, Komatsu proposed a generalized version of Sylvester sums, referred to as weighted Sylvester sums \cite{KZTV,KZMV}. Additionally, he derived explicit results under certain special conditions \cite{KomAP,KomAP1}. For a fixed positive integer $m$, the interested reader may refer to Beck and Bardomero \cite{BBeck}, where the authors present a formula for the $m$th power sum of positive integers that can be represented in exactly $k$ ways as nonnegative integer combinations of $a$ and $b$, for any $k \geq 1$.

In this paper, we introduce a recursive formula for the Sylvester sums $S_m(a,b)$ for any nonnegative integer $m$. Our technique differs from the analytic approaches involving Bernoulli numbers found in the work of R{\o}dseth \cite{JR} and Tuenter \cite{HJT}, offering instead a direct and elementary combinatorial derivation.

The key result~\cite[Corollary~17]{D19} underlying our work was originally studied by Barlow~\cite[pp.~323--325]{B18}, and appears to be due to Popoviciu~\cite{P53}; see ~\cite[Chapter~1, Theorem~1.5]{BR07}. An equivalent version of this result was later derived using generating functions by Tripathi~\cite{ATax}. Our work is motivated by a recent proof of this identity given by Binner~\cite[Corollary~17]{D19}, who obtained it as a corollary to his main theorem, which we recall here as Corollary~\ref{D3}.

\begin{corollary}[Binner 2020]
\label{D3}
Let $a$ and $b$ be relatively prime positive integers and $n$ be some positive integer. Then, the number of nonnegative integer solutions, denoted by  $N(a,b;n)$, of the equation $ax+by=n$, is given by
\begin{equation*}
    N(a,b;n)=1+\frac{n-aa_1-bb_1}{ab},
\end{equation*}
where the symbols $a_1$ and $b_1$ denote the remainder when $na^{-1}$ is divided by $b$ and $nb^{-1}$ divided by $a$ respectively and $a^{-1}$ and $b^{-1}$ denote the modular inverse of $a$ with respect to $b$ and $b$ with respect to $a$.
\end{corollary}
\begin{remark}
    Corollary \ref{D3} remains valid for case $n=0$.
\end{remark}

For any nonnegative integer $n$ such that $n<ab$, Corollary \ref{D3} implies that $N(a,b;n)$ is either $0$ or $1$. It implies that $n$ is either a nonrepresentable number or has a unique representation. More precisely, when $N(a,b;n)=0$, Corollary \ref{D3} ensures that there exist $0\leq a_1 \leq b-1$ and $0\leq b_1 \leq a-1$ such that 
\begin{equation*}
    n=aa_1+bb_1-ab.
\end{equation*}
On the other hand, when $N(a,b;n)=1$, Corollary \ref{D3} ensures that there exist $0\leq a_1 \leq b-1$ and $0\leq b_1 \leq a-1$ such that 
\begin{equation*}
    n=aa_1+bb_1.
\end{equation*}
Thus for any nonnegative integer $n$ such that $0\leq n<ab$, the integer $n$ can be expressed as 
\begin{equation*}
    n=aa_1+bb_1-ab\chi_{a_1,b_1}
\end{equation*}
for some $a_1$ and $b_1$ such that $0\leq a_1 \leq b-1$ and $0\leq b_1 \leq a-1$, where
\begin{displaymath}
    \chi_{a_1,b_1} = \begin{cases}
    1, & \text{if } N(a,b;n) = 0;\\
    0, & \text{if } N(a,b;n) = 1.
    \end{cases}
\end{displaymath}
Since the set of nonrepresentable numbers is contained in the set $\{n\in \mathbb{Z}: 0\leq n <ab  \}$, the cardinality of the set $\NR(a,b)$ can therefore be expressed in terms of $\chi_{a_1,b_1}$ as
\begin{equation}
\label{sn}
    |\NR(a,b)|= \sum_{\substack{ 0\leq a_1\leq b-1\\ 0\leq b_1\leq a-1}}\chi_{a_1,b_1}.
\end{equation}

To formalize these ideas, we state them as Theorem \ref{main}, which we prove in Section \ref{sec22} using Lemma \ref{SY1}, Theorem \ref{SY2} and Corollary \ref{D3}. Theorem \ref{main} also explicitly defines an equivalent form of $\chi_{a_1,b_1}$ explicitly in terms of $a_1$ and $b_1$, where $0\leq a_1\leq b-1$ and $0\leq b_1\leq a-1$.  
\begin{theorem}
\label{main}
   Let $a$ and $b$ be relatively prime positive integers, and let $a_1$ and $b_1$ be as defined in Corollary \ref{D3}. For nonnegative integer $n$, the set $\{n:0\leq n\leq ab-1 \}$ is the same as the set $\{aa_1+bb_1-ab\chi_{a_1,b_1}: 0\leq a_1 \leq b-1, 0\leq b_1 \leq a-1  \}$ where 
   \begin{displaymath}
    \chi_{a_1,b_1} = 
    \begin{cases} 
        1, & \quad \text{if } aa_1 + bb_1 - ab > 0; \\  
        0, & \quad \text{if } aa_1 + bb_1 - ab < 0.  
    \end{cases}
\end{displaymath}
\end{theorem}
Applying Theorem \ref{main}, we have
\begin{equation}
\label{ET}
    \sum_{n=0}^{ab-1} n^m= \sum_{\substack{0\leq a_1 \leq b-1 \\ 0\leq b_1 \leq a-1}} (aa_1+bb_1-ab\chi_{a_1,b_1})^m.
\end{equation}

For a given positive integer $m$, our aim is to prove our main result, Theorem \ref{mt}, which provides a formula to calculate the Sylvester sum $S_{m-1}$. We expand the right-hand side of Equation \eqref{ET}. To simplify this expression and isolate the desired Sylvester sum, we utilize a combinatorial identity presented in Proposition \ref{ml2}, which we prove in Section \ref{prop}. On the left-hand side, we express the power sum of the first $n$ natural numbers using Bernoulli's formula \cite[Eq.\ 1.1, p.\ 1]{BBer}, as follows:
\begin{theorem}[Bernoulli 1713]
\label{bernoulli}
For natural numbers $n$ and $k$
    \begin{equation*}
    \sum_{i=1}^{n}i^k=\frac{1}{k+1}\sum_{j=0}^{k}\binom{k+1}{j}B_j
n^{k+1-j},     
\end{equation*}
where $\binom{k}{j}$ is the binomial coefficient 
$$\binom{k}{j}=\frac{k(k-1)(k-2)\cdots (k-j+1)}{j!},$$
and $B_j$ is the number, known as the Bernoulli number, determined by the recursive formula
\begin{equation}
\label{BNRF}
    \sum_{j=0}^{k} \binom{k+1}{j} B_j = k+1, \quad \text{for } k = 0, 1, 2, \dots.
\end{equation}
\end{theorem}
   \begin{remark}
    The first three Bernoulli numbers, calculated using Equation \eqref{BNRF}, are:
    \begin{equation}
        \label{IBN}
        B_0=1, \quad B_1=\frac{1}{2}, \quad B_{2}=\frac{1}{6}.
    \end{equation}
\end{remark}

Next, we present two corollaries of Theorem \ref{main}, namely Corollary \ref{OC} and Corollary \ref{OT}, which we prove in Section \ref{sec22}. Corollary \ref{OC} is used to establish Corollary \ref{OT}, and Corollary \ref{OT}, in turn, is used to prove Proposition \ref{ml1'}. 

We begin by defining the following notation: 
\begin{Notation}
    For a positive integer $m$, let $\NR^m(a,b)$ denote the set consisting of $m^{th}$ powers of all nonrepresentable numbers. That is
    \begin{equation}
        \label{NRm}
        \NR^m(a,b):=\{n^m:n\in \NR(a,b) \}.
    \end{equation}
\end{Notation}

\begin{corollary}
\label{OC}
For positive integer $m$ and relatively prime positive integers $a$ and $b$,
\begin{equation*}
    \NR^m(a,b)=\{(aa_1+bb_1-ab)^m\chi_{a_1,b_1}: 0\leq a_1\leq b-1, 0\leq b_1\leq a-1 \}\setminus \{0\}.  
  \end{equation*}
\end{corollary}

\begin{remark} 
Observe that $\NR^m(a,b)$ is a finite subset of the set of natural numbers. \end{remark}

\begin{corollary}
  \label{OT}  
For nonnegative integer $m$, the $m^{th}$ power sum of all the elements of $\NR(a,b)$, denoted as $S_m(a,b)$, can be expressed as, 
\begin{equation*}
 S_m (a,b)=\sum_{\substack{ 0\leq a_1\leq b-1\\ 0\leq b_1\leq a-1}}(aa_1+bb_1-ab)^{m}\chi_{a_1,b_1}.   
\end{equation*}
\end{corollary}
\begin{remark}
    For brevity, we sometimes write $S_m (a,b):=S_m$. 
\end{remark}

The following Lemma \ref{ml1}, whose proof is provided in Section \ref{prop} using the snake oil method \cite[Sec.\ 4.3, p.\ 118]{Wilf}, is instrumental in proving Proposition \ref{ml2}.
\begin{lemma}
    \label{ml1}
For positive integers $n$ and $m$ with $n\geq m$
\begin{equation*}
    \sum_{i=1}^{m}(-1)^{i+1}\binom{n}{i}\binom{n-i}{m-i}=\binom{n}{m}.
\end{equation*}
    \end{lemma}

\begin{proposition}
\label{ml1'}
For nonnegative integer $n$
    \label{ml2}
    \begin{equation*}
         \sum_{\substack{0\leq a_1 \leq b-1 \\ 0\leq b_1 \leq a-1}} (aa_1+bb_1)^{n} \chi_{a_1,b_1}=\sum_{i=0}^{n} \binom{n}{i}a^{i}b^{i}S_{n-i}.
  \end{equation*}  
\end{proposition}
Applying Theorem \ref{main} and Proposition \ref{ml2}, we establish our main result, Theorem \ref{mt}, which is presented in Section \ref{sec3}. 

\begin{theorem}
    \label{mt}
    Let $a,b\in \mathbb{N}$ such that $\gcd(a,b)=1$, then for $m\geq1$
    \begin{align*}
         mabS_{m-1}&= \sum_{\substack{0\leq a_1 \leq b-1 \\ 0\leq b_1 \leq a-1}} (aa_1+bb_1)^m-mab\sum_{j=1}^{m-1}\binom{m-1}{j}a^{j}b^{j}S_{m-1-j} \\ &\quad +\sum_{i=2}^{m}(-1)^{i}\binom{m}{i}a^{i}b^{i}\sum_{j=0}^{m-i}\binom{m-i}{j}a^{j}b^{j}S_{m-i-j}\\& \quad\quad-\frac{1}{m+1}\sum_{j=0}^{m}\binom{m+1}{j}B_{j}(ab-1)^{m+1-j}. 
    \end{align*}
    \end{theorem}

From Theorem \ref{mt}, we derive two corollaries-Corollary \ref{I2} and Corollary \ref{I3}-which provide explicit expressions for the Sylvester number $S_0(a,b)$ and the sum of nonrepresentable numbers $S_1(a,b)$. The proofs of these corollaries are presented in Section \ref{sec3}.
\begin{corollary}
\label{I2}
\label{DNM2}
    Let $a,b\in \mathbb{N}$ such that $\gcd(a,b)=1$, then $$S_0 (a,b)=\frac{1}{2}(a-1)(b-1).$$
\end{corollary}

\begin{corollary}
\label{I3}
\label{DNM3}
    Let $a,b\in \mathbb{N}$ such that $\gcd(a,b)=1$, then $$S_1 (a,b)=\frac{1}{12}(a-1)(b-1)(2ab-a-b-1).$$
\end{corollary}

Furthermore, Theorem \ref{DNM1} establishes a criterion for determining whether a given integer $n\in \NR(a,b)$. To proceed, we first introduce the following notation:
Let $m$ be a nonnegative integer, and let $a, b \in \mathbb{N}$ with $gcd(a,b)=1$. Further, suppose $b=aq_{0}+r_{0}$ where $q_0$ and $r_0$ are nonnegative integers with $0< r_0<a$. We define the following symbols:
\begin{itemize}
    \item $r:= m$ (mod $a$) such that $ 0\leq r<a$.
    \item $r_{m}:= r$ (mod $r_0$) such that $ 0\leq r_m<r_0$.
    \item $k_{m}\equiv -a_{r_0}^{-1}r_m$ (mod $r_0$) such that $ 0\leq k_m<r_0$.
\end{itemize}
 
\begin{theorem}
\label{DNM1}
   With the notation above, a nonnegative integer $m$ can be expressed in the form $ax+by$ for some nonnegative integers $x$ and $y$ if and only if
$$\left\lfloor \frac{m}{a} \right\rfloor \geq q_{0}\frac{m}{b}+k_m.$$
\end{theorem}

Theorem \ref{DNM1} also generalizes Lemma 6 and Lemma 8 in \cite{D21}. We recall these lemmas below as Lemma \ref{D1} and Lemma \ref{D2}, respectively.
\begin{lemma}[Binner 2021]
\label{D1}
    For any $a\in \mathbb{N}$, a number $m\in \mathbb{N}$ can be expressed in the form $ax+(a+1)y$ for nonnegative integers $x$ and $y$ if and only if $\lfloor \frac{m}{a}\rfloor\geq \frac{m}{a+1}.$
\end{lemma}

\begin{lemma}[Binner 2021]
\label{D2}
    For any odd $a\in \mathbb{N}$, a number $m\in \mathbb{N}$ can be expressed in the form $ax+(a+2)y$ for nonnegative integers $x$ and $y$ if and only if one of the following statements hold:
    \begin{enumerate}
       \item $m \Mod a$ is even and  $\lfloor \frac{m}{a} \rfloor\geq \frac{m}{a+2}$.
\item $m \Mod a$ is odd and  $\lfloor \frac{m}{a}\rfloor \geq \frac{m}{a+2}+1$.
    \end{enumerate}
\end{lemma}

\section{Proof of Theorem \ref{main}}
\label{sec22}
\begin{proof}
For any integer $n$ such that $0\leq n \leq ab-1$, suppose that $N(a,b;n)\geq 2$. By Corollary \ref{D3}, it follows that $n\geq ab$, which contradicts the assumption that $0\leq n \leq ab-1$. Therefore, $N(a,b;n)<2$. Applying Lemma \ref{SY1} and Theorem \ref{SY2}, we conclude that $N(a,b;n)=0$ or $N(a,b;n)=1$. If $N(a,b;n)=0$, then using Corollary \ref{D3}, we have $n=aa_1+bb_1-ab$ for some $0\leq a_1\leq b-1$ and $0\leq b_1\leq a-1$. Since $n$ cannot be zero in this case, it implies $n>0$. Thus,
\begin{equation}
\label{NRN}
    aa_1+bb_1-ab>0.
\end{equation}

Similarly, when $N(a,b;n)=1$, we have $n=aa_1+bb_1$ for some $0\leq a_1\leq b-1$ and $0\leq b_1\leq a-1$ . Given that $n<ab$, it implies
\begin{equation}
\label{RN}
    aa_1+bb_1-ab<0.
\end{equation}
Using Equations \eqref{NRN} and \eqref{RN}, every integer $n$ such that $0\leq n\leq ab-1$ can be expressed as $$n=aa_1+bb_1-ab\chi_{a_1,b_1},$$ where 
\begin{displaymath}
    \chi_{a_1,b_1} = 
    \begin{cases} 
        1, & \quad \text{if } aa_1 + bb_1 - ab > 0; \\  
        0, & \quad \text{if } aa_1 + bb_1 - ab < 0.  
    \end{cases}
\end{displaymath}
Since the cardinality of the set $\{n:0\leq n \leq ab-1 \}$ is $ab$, which is the same as the cardinality of the set $\{aa_1+bb_1-ab\chi_{a_1,b_1}: 0\leq a_1\leq b-1, 0\leq b_1\leq a-1  \}$. Therefore, the result follows.
    \end{proof}

\begin{proof}[Proof of Corollary \ref{OC}]
Let us denote $S=\{(aa_1+bb_1-ab)^m\chi_{a_1,b_1}: 0\leq a_1\leq b-1, 0\leq b_1\leq a-1 \}\setminus \{0\}.$
Consider any $n\in \NR^m(a,b)$. By Equation \eqref{NRm}, this implies $n=k^m$ for some $k\in \NR(a,b)$. Consequently, $N(a,b;k)=0$, by Corollary \ref{D3}, there exist some $a_1$ with $0\leq a_1\leq b-1$ and $b_1$ with $0\leq b_1\leq a-1$
such that
$$k=aa_1+bb_1-ab>0.$$
Since $aa_1+bb_1-ab>0$, it follows from Theorem \ref{main} that the corresponding $\chi_{a_1,b_1}=1$. Thus, $n$ can be expressed as
$$n=k^m \chi_{a_1,b_1}=(aa_1+bb_1-ab)^m\chi_{a_1,b_1}.$$
This shows that $\NR^m(a,b) \subseteq S.$
Now suppose $n\in S$. This implies that $n=(aa_1+bb_1-ab)^m\chi_{a_1,b_1}>0$, for some $a_1$ with $0\leq a_1\leq b-1$ and $b_1$ with $0\leq b_1\leq a-1$.
Since $(aa_1+bb_1-ab)^m\chi_{a_1,b_1}>0$, it follows that $\chi_{a_1,b_1}=1$. By Theorem \ref{main}, we have $k=aa_1+bb_1-ab>0$. From Corollary \ref{D3}, it follows that $N(a,b;k)=0$, and thus $n\in \NR^m(a,b)$. Therefore, $S\subseteq \NR^m(a,b) .$
\end{proof}

\begin{proof}[Proof of Corollary \ref{OT}]
    For $m=0$, the proof follows using Equations \eqref{snm} and \eqref{sn}. When $m>0$, the proof follows using Corollary \ref{OC}. 
    \end{proof}

\section{Proof of Proposition \ref{ml1'}}
\label{prop}
To establish Proposition \ref{ml1'}, we use a substitution based on Lemma \ref{ml1}, simplifying the calculations. Therefore, we begin by first proving Lemma \ref{ml1}.
\begin{proof}[Proof of Lemma \ref{ml1}]
 First, observe that for $m\geq 1$,
\begin{equation*}
    \sum_{i=1}^{m}(-1)^{i+1}\binom{n}{i}\binom{n-i}{m-i}= \binom{n}{m}+\sum_{i=0}^{m}(-1)^{i+1}\binom{n}{i}\binom{n-i}{m-i}.
\end{equation*}
  Therefore, to prove the required claim, it suffices to show that for all $m\geq 1$,
\begin{equation*}
    \sum_{i=0}^{m}(-1)^{i+1}\binom{n}{i}\binom{n-i}{m-i}=0.
\end{equation*}

 Define 
 \begin{align*}
 f(m)&=\sum_{i=0}^{m}(-1)^{i+1}\binom{n}{i}\binom{n-i}{m-i}.
      \end{align*}
      \begin{align*}
       \intertext{Using the fact that $\binom{n}{k}=0$ when either $k>n$ or $k<0$, we can rewrite $f(m)$ as follows:}
 f(m)=\sum_{i=0}^{n}(-1)^{i+1}\binom{n}{i}\binom{n-i}{m-i}.
      \end{align*}
 Let $$F(x)=\sum_{m=0}^{\infty}f(m)x^m,$$ be the generating function for $f(m)$. Notice that
\begin{align*}
    \sum_{m=0}^{\infty}f(m)x^m&= \sum_{m=0}^{\infty}\sum_{i=0}^{n}(-1)^{i+1}\binom{n}{i}\binom{n-i}{m-i}x^m\\
&=\sum_{i=0}^{n} (-1)^{i+1}\binom{n}{i}x^{i}\sum_{m=0}^{\infty}\binom{n-i}{m-i}x^{m-i}\\
&=\sum_{i=0}^{n} (-1)^{i+1}\binom{n}{i}x^{i}\sum_{m=i}^{n}\binom{n-i}{m-i}x^{m-i}.\\
\intertext{Substitute $m-i=k$, we have}
&=\sum_{i=0}^{n}  (-1)^{i+1}\binom{n}{i}x^{i}\sum_{k=0}^{n-i}\binom{n-i}{k}x^{k}\\
&=\sum_{i=0}^{n} (-1)^{i+1}\binom{n}{i}x^{i}(1+x)^{n-i}\\
&=-\sum_{i=0}^{n} (-1)^{i}\binom{n}{i}x^{i}(1+x)^{n-i}\\
&=-(1+x-x)^n\\
&=-1.
\end{align*}
This implies that $f(m)=0$ for $m\geq1$.
  \end{proof}

Now, we proceed with the proof of Proposition \ref{ml1'}. We establish the result using the Principle of Mathematical Induction. For $n=1$ note that 
\begin{align*}
    (aa_1+bb_1) \chi_{a_1,b_1}&= (aa_1+bb_1-ab) \chi_{a_1,b_1}+ab\chi_{a_1,b_1}.\\
    \intertext{Running summation over $a_1$ and $b_1$ on both the sides}
    \sum_{\substack{0\leq a_1 \leq b-1 \\ 0\leq b_1 \leq a-1}} (aa_1+bb_1) \chi_{a_1,b_1}&=\sum_{\substack{0\leq a_1 \leq b-1 \\ 0\leq b_1 \leq a-1}}  (aa_1+bb_1-ab) \chi_{a_1,b_1}+ab\sum_{\substack{0\leq a_1 \leq b-1 \\ 0\leq b_1 \leq a-1}} \chi_{a_1,b_1}\\
    &=S_{1}+abS_{0} \text{ \text{ }\text{ }  \text{ }      [By applying Corollary \ref{OT}]} \\
    &=\sum_{i=0}^{1}\binom{1}{i}a^{i}b^{i}S_{1-i}.
\end{align*}
Hence, the result holds for the $n=1$ case.

Now assume that the result holds for $n=k$, that is, we assume the following:
\begin{equation*}
    \sum_{\substack{0\leq a_1 \leq b-1 \\ 0\leq b_1 \leq a-1}} (aa_{1}+bb_{1})^k \chi_{a_1,b_1}=\sum_{i=0}^{k} \binom{k}{i}a^{i}b^{i}S_{k-i}.
\end{equation*}
Finally, we prove that the result holds for the $n=k+1$ case. That is, we need to show the following 
\begin{equation*}
   \sum_{\substack{0\leq a_1 \leq b-1 \\ 0\leq b_1 \leq a-1}} (aa_{1}+bb_{1})^{k+1} \chi_{a_1,b_1}=\sum_{i=0}^{k+1} \binom{k+1}{i}a^{i}b^{i}S_{k+1-i}. 
\end{equation*}
 Consider
      \begin{align*}
     (aa_{1}+bb_{1})^{k+1} \chi_{a_1,b_1}&= \begin{multlined}[t]\sum_{i=0}^{k+1}(-1)^{i}\binom{k+1}{i}(aa_1+bb_1)^{k+1-i}a^ib^i\chi_{a_1,b_1}\\[2ex]-\sum_{i=1}^{k+1}(-1)^{i}\binom{k+1}{i}(aa_1+bb_1)^{k+1-i}a^ib^i\chi_{a_1,b_1}\end{multlined}\\
    &=(aa_1+bb_1-ab)^{k+1}\chi_{a_1,b_1}\\&\quad-\sum_{i=1}^{k+1}(-1)^{i}\binom{k+1}{i}(aa_1+bb_1)^{k+1-i}a^ib^i\chi_{a_1,b_1}.\\
     \intertext{Running summation over $a_1$ and $b_1$ on both the sides}
\sum_{\substack{0\leq a_1 \leq b-1 \\ 0\leq b_1 \leq a-1}} (aa_{1}+bb_{1})^{k+1} \chi_{a_1,b_1} &=\begin{multlined}[t]
    \sum_{\substack{0\leq a_1 \leq b-1 \\ 0\leq b_1 \leq a-1}} (aa_1+bb_1-ab)^{k+1}\chi_{a_1,b_1}\\[2ex]-\sum_{\substack{0\leq a_1 \leq b-1 \\ 0\leq b_1 \leq a-1}} \sum_{i=1}^{k+1}(-1)^{i}\binom{k+1}{i}(aa_1+bb_1)^{k+1-i}a^ib^i\chi_{a_1,b_1}\end{multlined}\\
    \end{align*}
    \begin{align*}
     &=S_{k+1}-\sum_{i=1}^{k+1}(-1)^{i}\binom{k+1}{i}a^ib^i\sum_{\substack{0\leq a_1 \leq b-1 \\ 0\leq b_1 \leq a-1}} (aa_1+bb_1)^{k+1-i}\chi_{a_1,b_1}.\\
     \intertext{Now using our assumption substituting the value for $$\sum_{\substack{0\leq a_1 \leq b-1 \\ 0\leq b_1 \leq a-1}} (aa_1+bb_1)^{k+1-i}\chi_{a_1,b_1},$$ we obtain}
     &=S_{k+1}-\sum_{i=1}^{k+1}(-1)^{i}\binom{k+1}{i}a^ib^i\sum_{j=0}^{k+1-i}\binom{k+1-i}{j}a^{j}b^{j}S_{k+1-i-j}\\
    &=\begin{multlined}[t]
       S_{k+1}+ab\binom{k+1}{1} \sum_{j=0}^{k}\binom{k}{j}a^{j}b^{j}S_{k-j} -a^2b^2\binom{k+1}{2} \sum_{j=0}^{k-1}\binom{k-1}{j}a^{j}b^{j}S_{k-1-j} +\\[2ex]\cdots+(-1)^ka^kb^k\binom{k+1}{k} \sum_{j=0}^{1}\binom{k}{j}a^{j}b^{j}S_{1-j}+(-1)^{k+1}a^{k+1}b^{k+1}\binom{k+1}{k+1} \binom{0}{0}S_{0}
    \end{multlined}
    \end{align*}
    \begin{align*}
    &=\begin{multlined}[t]
        S_{k+1}+ab \binom{k+1}{1}\binom{k}{0} S_k+a^2b^2\left( \binom{k+1}{1}\binom{k}{1}-\binom{k+1}{2}\binom{k-1}{0} \right)S_{k-1}+\\[2ex]\cdots+a^{k+1}b^{k+1}\left( \binom{k+1}{1}\binom{k}{k}-\binom{k+1}{2}\binom{k-1}{k-1}+\cdots+ (-1)^{k+1}\binom{k+1}{k+1}\binom{0}{0} \right)S_{0}
    \end{multlined}\\
    &=S_{k+1}+\sum_{i=1}^{k+1}a^ib^i\sum_{j=1}^{i}(-1)^{j+1}\binom{k+1}{j}\binom{k+1-j}{i-j}S_{k+1-i}.
\intertext{Now using Lemma \ref{ml1} substituting the value of $$\sum_{j=1}^{i}(-1)^{j+1}\binom{k+1}{j}\binom{k+1-j}{i-j}=\binom{k+1}{i},$$ we obtain}
&=S_{k+1}+\sum_{i=1}^{k+1}a^ib^i\binom{k+1}{i}S_{k+1-i}\\
&=\sum_{i=0}^{k+1}a^ib^i\binom{k+1}{i}S_{k+1-i}.
    \end{align*}
 
    \section{Proof of Theorem \ref{mt}}
\label{sec3}
Using Theorem \ref{main}, we have
  \begin{align*}
\sum_{n=0}^{ab-1} n^m&= \sum_{\substack{0\leq a_1 \leq b-1 \\ 0\leq b_1 \leq a-1}} (aa_1+bb_1-ab\chi_{a_1,b_1})^m\\
&=  \sum_{\substack{0\leq a_1 \leq b-1 \\ 0\leq b_1 \leq a-1}}\sum_{i=0}^{m}(-1)^i\binom{m}{i}(aa_1+bb_1)^{m-i}a^ib^i \chi_{a_1,b_1}^i.\\
\intertext{We define $\chi_{a_1,b_1}^0:=1$. Furthermore, note that for $i>0$, $\chi_{a_1,b_1}^i= \chi_{a_1,b_1}$. Using this, we obtain}
&= \sum_{\substack{0\leq a_1 \leq b-1 \\ 0\leq b_1 \leq a-1}}(aa_1+bb_1)^m+ \sum_{\substack{0\leq a_1 \leq b-1 \\ 0\leq b_1 \leq a-1}}\sum_{i=1}^{m}(-1)^i\binom{m}{i}(aa_1+bb_1)^{m-i}a^ib^i \chi_{a_1,b_1}\\
&= \sum_{\substack{0\leq a_1 \leq b-1 \\ 0\leq b_1 \leq a-1}}(aa_1+bb_1)^m+ \sum_{i=1}^{m}(-1)^i\binom{m}{i}a^ib^i\sum_{\substack{0\leq a_1 \leq b-1 \\ 0\leq b_1 \leq a-1}}(aa_1+bb_1)^{m-i} \chi_{a_1,b_1}.
\intertext{Now using Proposition \ref{ml2}}
&= \sum_{\substack{0\leq a_1 \leq b-1 \\ 0\leq b_1 \leq a-1}}(aa_1+bb_1)^m+ \sum_{i=1}^{m}(-1)^i\binom{m}{i}a^ib^i\sum_{j=0}^{m-i}\binom{m-i}{j}a^jb^jS_{m-i-j}\\
&=\begin{multlined}[t] \sum_{\substack{0\leq a_1 \leq b-1 \\ 0\leq b_1 \leq a-1}}(aa_1+bb_1)^m-\binom{m}{1}ab\sum_{j=0}^{m-1}\binom{m-1}{j}a^jb^jS_{m-1-j} \\[2ex]+ \sum_{i=2}^{m}(-1)^i\binom{m}{i}a^ib^i\sum_{j=0}^{m-i}\binom{m-i}{j}a^jb^jS_{m-i-j}
\end{multlined}\\
&=\begin{multlined}[t]
    \sum_{\substack{0\leq a_1 \leq b-1 \\ 0\leq b_1 \leq a-1}}(aa_1+bb_1)^m-mab\binom{m-1}{0}S_{m-1}-mab\sum_{j=1}^{m-1}\binom{m-1}{j}a^jb^jS_{m-1-j} \\[2ex]+ \sum_{i=2}^{m}(-1)^i\binom{m}{i}a^ib^i\sum_{j=0}^{m-i}\binom{m-i}{j}a^jb^jS_{m-i-j}.
\end{multlined} 
\intertext{That is}
\sum_{n=0}^{ab-1}n^m &=\begin{multlined}[t]
    \sum_{\substack{0\leq a_1 \leq b-1 \\ 0\leq b_1 \leq a-1}}(aa_1+bb_1)^m-mabS_{m-1}-mab\sum_{j=1}^{m-1}\binom{m-1}{j}a^jb^jS_{m-1-j} \\[2ex]+ \sum_{i=2}^{m}(-1)^i\binom{m}{i}a^ib^i\sum_{j=0}^{m-i}\binom{m-i}{j}a^jb^jS_{m-i-j}.
\end{multlined} 
\intertext{That is}
mabS_{m-1} &=\begin{multlined}[t]
    \sum_{\substack{0\leq a_1 \leq b-1 \\ 0\leq b_1 \leq a-1}}(aa_1+bb_1)^m-mab\sum_{j=1}^{m-1}\binom{m-1}{j}a^jb^jS_{m-1-j} \\[2ex]+ \sum_{i=2}^{m}(-1)^i\binom{m}{i}a^ib^i\sum_{j=0}^{m-i}\binom{m-i}{j}a^jb^jS_{m-i-j}-\sum_{n=0}^{ab-1}n^m.
    \end{multlined} 
\intertext{Finally, applying Theorem \ref{bernoulli}, we substitute the corresponding value for $$\sum_{n=0}^{ab-1}n^m,$$ we obtain}
 mabS_{m-1}&= \sum_{\substack{0\leq a_1 \leq b-1 \\ 0\leq b_1 \leq a-1}} (aa_1+bb_1)^m-mab\sum_{j=1}^{m-1}\binom{m-1}{j}a^{j}b^{j}S_{m-1-j} \\&\quad +\sum_{i=2}^{m}(-1)^{i}\binom{m}{i}a^{i}b^{i}\sum_{j=0}^{m-i}\binom{m-i}{j}a^{j}b^{j}S_{m-i-j}\\& \quad\quad-\frac{1}{m+1}\sum_{j=0}^{m}\binom{m+1}{j}B_{j}(ab-1)^{m+1-j}. 
    \end{align*}

\begin{proof}[Proof of Corollary \ref{DNM2}]
In Theorem \ref{mt} putting $m=1$ and simplifying
\begin{align*}
        abS_{0}&=\begin{multlined}[t]
            \sum_{\substack{0\leq a_1 \leq b-1 \\ 0\leq b_1 \leq a-1}} (aa_1+bb_1)-ab\sum_{j=1}^{0}\binom{0}{j}a^{j}b^{j}S_{0} \\[2ex]+\sum_{i=2}^{1}(-1)^{i}\binom{1}{i}a^{i}b^{i}\sum_{j=0}^{1-i}\binom{1-i}{j}a^{j}b^{j}S_{1-i-j}-\frac{1}{2}\sum_{j=0}^{1}\binom{2}{j}B_{j}(ab-1)^{2-j}. 
         \end{multlined}
         \intertext{Now using $\binom{n}{k}=0$ when $k\geq n$ we have}
         &=\sum_{\substack{0\leq a_1 \leq b-1 \\ 0\leq b_1 \leq a-1}} (aa_1+bb_1)-\frac{1}{2}B_0(ab-1)^2-B_1(ab-1).\\
         \intertext{Using Equation \eqref{IBN} substituting the values $B_0=1$ and $B_1=\frac{1}{2}$, we obtain}
        &= a\sum_{\substack{0\leq a_1 \leq b-1 \\ 0\leq b_1 \leq a-1}} a_1+b\sum_{\substack{0\leq a_1 \leq b-1 \\ 0\leq b_1 \leq a-1}} b_1-\frac{1}{2}ab(ab-1)\\ 
        &=a^2\sum_{0\leq a_1 \leq b-1}a_1+b^2\sum_{0\leq b_1 \leq a-1}b_1-\frac{1}{2}ab(ab-1)\\
         &= \frac{1}{2}a^2b(b-1)+\frac{1}{2}b^2a(a-1)-\frac{1}{2}ab(ab-1)\\
         &=\frac{1}{2}ab(ab-a-b+1).
        \end{align*}
        Therefore $$ S_0=\frac{1}{2}(a-1)(b-1).$$
\end{proof}

\begin{proof}[Proof of Corollary \ref{DNM3}]
In Theorem \ref{mt} putting $m=2$ and simplifying
\begin{align*}
        2abS_{1}&=\begin{multlined}[t]
            \sum_{\substack{0\leq a_1 \leq b-1 \\ 0\leq b_1 \leq a-1}} (aa_1+bb_1)^2-2ab\sum_{j=1}^{1}\binom{1}{j}a^{j}b^{j}S_{1-j} \\[2ex]+\sum_{i=2}^{2}(-1)^{i}\binom{2}{i}a^{i}b^{i}\sum_{j=0}^{2-i}\binom{2-i}{j}a^{j}b^{j}S_{2-i-j}-\frac{1}{3}\sum_{j=0}^{2}\binom{3}{j}B_{j}(ab-1)^{3-j}  
         \end{multlined}\\
         &=\sum_{\substack{0\leq a_1 \leq b-1 \\ 0\leq b_1 \leq a-1}} (aa_1+bb_1)^2-2a^2b^2S_{0}+a^2b^2S_{0}-\frac{1}{3}B_0(ab-1)^3-B_1(ab-1)^2\\&\quad\quad-B_2(ab-1).\\
         \intertext{Using Equation \eqref{IBN} substituting the values $B_0, B_1$ and $B_2$, we obtain}
         &=a^2\sum_{\substack{0\leq a_1 \leq b-1 \\ 0\leq b_1 \leq a-1}} a_{1}^2+2ab\sum_{\substack{0\leq a_1 \leq b-1 \\ 0\leq b_1 \leq a-1}} a_{1}b_1+b^2\sum_{\substack{0\leq a_1 \leq b-1 \\ 0\leq b_1 \leq a-1}} b_{1}^2-a^2b^2S_{0}-\frac{1}{3}(ab-1)^3\\&\quad \quad \quad -\frac{1}{2}(ab-1)^2 -\frac{1}{6}(ab-1)\\
         &=a^3\sum_{0\leq a_1 \leq b-1}a_{1}^2+a^2b(a-1)\sum_{0\leq a_1 \leq b-1}a_{1}+b^3\sum_{0\leq b_1 \leq a-1}b_{1}^2-a^2b^2S_{0}\\&\quad\quad\quad-\frac{1}{6}ab(ab-1)(2ab-1)\\
             &=\frac{1}{6}a^3b(b-1)(2b-1)+\frac{1}{2}a^2b^2(a-1)(b-1)+\frac{1}{6}b^3a(a-1)(2a-1)-a^2b^2S_{0}\\ &\quad \quad \quad -\frac{1}{6}ab(ab-1)(2ab-1)\\
              &=\frac{1}{6}a^3b(b-1)(2b-1)+a^2b^2S_0+\frac{1}{6}b^3a(a-1)(2a-1)-a^2b^2S_{0}\\&\quad\quad\quad-\frac{1}{6}ab(ab-1)(2ab-1)\\
         &=\frac{1}{6}a^3b(b-1)(2b-1)+\frac{1}{6}b^3a(a-1)(2a-1)-\frac{1}{6}ab(ab-1)(2ab-1).
          \intertext{That is,}
S_1&=\frac{1}{12}a^2(b-1)(2b-1)+\frac{1}{12}b^2(a-1)(2a-1)-\frac{1}{12}(ab-1)(2ab-1)\\
&=\frac{1}{12}(a-1)(b-1)(2ab-a-b-1).
             \end{align*}
        \end{proof}

\section{Proof of Theorem \ref{DNM1}}
\label{sec2}
 \begin{proof}
   Let $m$ be a number expressed as $ax+by$ for nonnegative integers $x$ and $y$. Without loss of generality, we can assume that $0\leq y <a$. Substituting $b=aq_0+r_0$, we have,
\begin{equation*}
    m=ax+(aq_0+r_0)y.
\end{equation*}
That is, 
\begin{equation}
\label{my}
    m=a(x+q_0y)+r_0y.
\end{equation}
Choose the smallest $0\leq k_m<r_0$ such that $0\leq r_0y-k_ma<a$. Then, Equation \eqref{my} can be written as 
\begin{equation*}
m=a(x+q_0y+k_m)+(r_0y-k_m a).
\end{equation*}

Note that by the division algorithm, there exist unique nonnegative integers $q$ and $r$ such that $m=aq+r$, where $0\leq r<a$. It follows that,
\begin{equation*}
    q=x+q_{0}y+k_{m} \hspace{0.4cm}\text{and} \hspace{0.4cm} r=r_{0}y-k_{m}a.
\end{equation*}
Solving for $x$ gives us
\begin{equation*}
    x=q-\frac{q_{0}r+k_{m}(aq_{0}+r_{0})}{r_{0}}.
\end{equation*}
Since $x\geq 0$, it follows that 
\begin{equation*}
     q\geq\frac{q_{0}r+k_{m}(aq_{0}+r_{0})}{r_{0}}.
\end{equation*}
Substituting $r=m-aq$, we obtain
\begin{equation*}
\begin{split}
 q&\geq \frac{q_{0}(m-aq)+k_{m}(aq_{0}+r_{0})}{r_{0}}\\
 &= \frac{q_{0}m-q_{0}aq+k_{m}(aq_{0}+r_{0})}{r_{0}}.\\
 \end{split}
\end{equation*}
That is,
\begin{equation*}
    q(aq_0+r_0)\geq q_0m+k_m(aq_0+r_0).
    \end{equation*}
    Solving further, we get
    \begin{equation*}
    \begin{split}
     q & \geq  q_{0}\frac{m}{aq_{0}+r_{0}}+k_{m}\\
&= q_{0}\frac{m}{b}+k_{m}.
   \end{split} 
\end{equation*}
    Since $q=\lfloor \frac{m}{a} \rfloor$, we obtain the following inequality:
 \begin{equation*}
   \left\lfloor \frac{m}{a} \right\rfloor \geq q_{0}\frac{m}{b}+k_{m}.
    \end{equation*}

Finally, note that if $r\equiv r_m$ (mod $r_0$) for some $0\leq r_m <r_0$, and since $r=r_0y-k_m a$, it follows that $-k_ma\equiv r_m$ (mod $r_0$). Therefore, $k_m\equiv -a_{r_0}^{-1}r_m$ (mod $r_0$).
Thus, the result holds.

Conversely, let $\lfloor \frac{m}{a}\rfloor \geq q_{0}\frac{m}{b}+k_{m}$, where $k_{m}$ is the same as defined earlier. Note that
\begin{equation*}
    \begin{split}
   \left \lfloor \frac{m}{a} \right\rfloor &\geq q_{0}\frac{m}{b}+k_{m}, \\
     q &\geq q_{0}\frac{m}{aq_{0}+r_{0}}+k_{m}, \\
     q(aq_{0}+r_{0}) &\geq q_{0}m+k_{m}(aq_{0}+r_{0}), \\
     q(aq_{0}+r_{0}) &\geq q_{0}(aq+r)+k_{m}(aq_{0}+r_{0}), \\
     qr_{0} &\geq q_{0}r+k_{m}(aq_{0}+r_{0}), \\
     q &\geq \frac{q_{0}r+k_{m}(aq_{0}+r_{0})}{r_0}.
      \end{split}
\end{equation*}

Further, note that
\begin{equation*}
    \frac{q_{0}r+k_{m}(aq_0+r_0)}{r_0}= \frac{q_{0}(r+k_{m}a)+k_{m}r_0}{r_0}.
\end{equation*}
Since $-ak_{m}\equiv r_m$(mod$r_0$) and $r\equiv r_m$ (mod$r_0$), it follows that $k_{m}a+r$ is divisible by $r_0$. Hence,
$$ \frac{r+k_{m}a}{r_0}, \frac{q_{0}(r+k_{m}a)+k_{m}r_0}{r_0} \in \mathbb{N}.$$
Now rewrite $m$ as follows
\begin{equation*}
\begin{split}
    m&=a\left( q-\frac{q_{0}r+k_{m}(aq_0+r_0)}{r_0}\right)+(aq_{0}+r_{0})\left(\frac{r+k_{m}a}{r_0}\right)\\
    &=a\left( q-\frac{q_{0}r+k_{m}(aq_0+r_0)}{r_0}\right)+b\left(\frac{r+k_{m}a}{r_0}\right)\\
    &=a\left( q-\frac{q_{0}(r+k_{m}a)+k_{m}r_0}{r_0}\right)+b\left(\frac{r+k_{m}a}{r_0}\right).
    \end{split}
\end{equation*}
Now take
\begin{equation*}
    x=\left( q-\frac{q_{0}(r+k_{m}a)+k_{m}r_0}{r_0}\right) \text{and}\ y=\left(\frac{r+k_{m}a}{r_0}\right),
\end{equation*}
we have
$$m=ax+by, \text{ where $x$ and $y$ are nonnegative integers}.$$
\end{proof}   

\section{Concluding Remarks}
In this article, we obtain a formula for calculating the power sum of nonrepresentable numbers in terms of $a$ and $b$, where $\gcd(a,b)=1$. The key result we use is the formula for the number of nonnegative integer solutions of the Diophantine equation $ax+by=n$, where $\gcd(a,b)=1$, given by Binner \cite{D19}. A natural question to ask is whether this result can be generalized to nonrepresentable numbers in terms of three variables $a$, $b$, and $c$, where $\gcd(a,b,c)=1$. In light of our proof technique, one may hope to use the formula for the number of nonnegative integer solutions of the Diophantine equation in three variables $ax+by+cz=n$, where $\gcd(a,b,c)=1$, also given by Binner \cite{D19}.

\section{Acknowledgment}
The authors appreciate the financial support and research facilities the Shiv Nadar Institute of Eminence provided. They also sincerely thank Dr.\ Damanvir Singh Binner for his invaluable guidance and insightful suggestions during the early stages of this research. His contributions were crucial in shaping the direction of this work, and his encouragement played a key role in facilitating the exploration of the research problems addressed in this study. The authors would like to thank the anonymous referee for valuable suggestions that helped improve the presentation and enrich the literature review of this paper.

\bibliographystyle{jis}
\bibliography{references}






\end{document}